\documentclass[11pt]{article}

\usepackage[a4paper,margin=1in]{geometry}
\usepackage{amsmath,amssymb,amsthm,mathtools}
\usepackage{enumitem}
\usepackage{microtype}
\usepackage{hyperref}
\hypersetup{colorlinks=true,linkcolor=black,citecolor=black,urlcolor=black}

\newcommand{\set}[1]{\left\{#1\right\}}

\newcommand{\ceil}[1]{\left\lceil #1\right\rceil}
\newcommand{\floor}[1]{\left\lfloor #1\right\rfloor}

\newtheorem{theorem}{Theorem}[section]
\newtheorem{lemma}[theorem]{Lemma}

\newtheorem{fact}[theorem]{Fact}
\theoremstyle{remark}
\newtheorem{remark}[theorem]{Remark}

\title{Hamiltonian cycles in 7-tough $(P_4\cup P_1)$-free graphs\thanks{This work is supported by the National Natural Science Foundation of China (Nos. 12371348 and 12201258) and the High-Quality Science and Technology Cultivation Project of Jiangsu Normal University (No. JSNUGZL2026069).}}
\author{Yong Lu,\thanks{Corresponding author.}\quad Qi Wu,\quad Qiannan Zhou\\[2mm]
\small School of Mathematics and Statistics, Jiangsu Normal University,\\
\small Xuzhou, Jiangsu 221116, People's Republic of China\\
\small E-mails: \texttt{luyong@jsnu.edu.cn}, \texttt{wuqimath@163.com}, \texttt{qnzhoumath@163.com}}
\date{}

\begin{document}
\maketitle

\begin{abstract}
Shan~[J. Graph Theory  (2026)] proved that every 23-tough $(P_4\cup P_1)$-free graph on at least three vertices is Hamiltonian. We improve this bound to 7 by replacing the final cut analysis in Shan's framework with an asymmetric separation criterion and a cograph covering lemma.

\medskip
\noindent\textbf{Keywords:} Hamiltonian cycle; toughness;  cograph; asymmetric separation criterion.

\noindent\textbf{AMS Subject Classification (2020):} 05C45.
\end{abstract}

\section{Introduction}

A graph is Hamiltonian if it contains a spanning cycle. A graph is called $H$-free if it contains no induced subgraph isomorphic to $H$. Classical sufficient conditions for Hamiltonicity, such as Dirac's minimum-degree theorem and the Chv\'atal--Erd\H{o}s connectivity condition, are expressed through local or semi-local graph parameters. Toughness, introduced by Chv\'atal in 1973~\cite{Chvatal1973}, measures instead the resistance of a graph to vertex separation. For a graph $G$ and a set $X\subseteq V(G)$, let $c(G-X)$ denote the number of connected components of $G-X$. The toughness of $G$ is defined as
\[
\tau(G)=\min\set{\frac{|X|}{c(G-X)}:X\subseteq V(G),\ c(G-X)\ge 2},
\]
with the convention that $\tau(G)=\infty$ if $G$ is complete. For a real number $t\ge 0$, a graph is $t$-tough if $\tau(G)\ge t$.

Every Hamiltonian graph is 1-tough. Chv\'atal conjectured that an absolute constant $t_0$ exists such that every $t_0$-tough graph is Hamiltonian. The conjecture remains open, and the constant, if it exists, is at least $9/4$ by the construction of Bauer, Broersma, and Veldman~\cite{BBV2000}. Surveys of toughness and its relation to Hamiltonicity may be found in~\cite{BBS2006,Broersma2015}.

One productive approach is to restrict the induced subgraphs that may occur. The class of $2K_2$-free graphs has been studied in~\cite{BPP2014,OtaSanka2022,Shan2020}, and related Hamiltonicity results for graphs excluding small linear forests appear in~\cite{GaoShan2022,Shan2021,ShanTanyel2025,ShiShan2022,XuLiZhou2024}. In these classes, toughness can be combined with a concrete structural decomposition and with path-cover arguments.

The graph $P_4\cup P_1$ is a natural boundary case. A $P_4$-free graph is a cograph, and Jung proved that every 1-tough cograph on at least three vertices is Hamiltonian~\cite{Jung1978}. Stronger spanning structures in cographs were considered in~\cite{EllinghamEtAl2020}. Allowing a single vertex to be isolated from an induced $P_4$ makes the problem substantially more difficult. Nikoghosyan conjectured that every 1-tough $(P_4\cup P_1)$-free graph is Hamiltonian~\cite{Nikoghosyan2013}, and the same class appeared as the remaining case in the forbidden-subgraph analysis of Li, Broersma, and Zhang~\cite{LiBroersmaZhang2016}. Very recently, Cao, Chen, and Zheng~\cite{CaoChenZheng2026} confirmed this conjecture for minimally 1-tough graphs with toughness exactly 1 whose every edge deletion lowers the toughness by proving that every such graph is one of $C_{4}$, $C_{5}$  or $C_{6}$.

Shan obtained the first absolute bound, proving that every 23-tough $(P_4\cup P_1)$-free graph is Hamiltonian~\cite{Shan2026}. The proof separates the vertices of degree below $n/4$, covers the induced cograph on those vertices by matched paths, compresses the paths to prescribed edges, and finds a suitable cycle in the remaining high-degree graph. The main numerical loss occurs when the compressed auxiliary graph has a small cut.

The present paper treats that cut directly. We first prove an asymmetric two-side criterion (Lemma \ref{lem:asymmetric}), in which one side of a separation has order roughly $n/8$ while the other has order roughly $n/4$. We next show that every prescribed cograph vertex set is contained in a cycle (Lemma \ref{lem:cograph-cover}). A small cut in the auxiliary graph then produces a larger prescribed cograph core (Lemma \ref{lem:enlargement}). Each enlargement adds at least
\[
\ceil{\frac n4}-\ceil{\frac n7}
\]
vertices. Two enlargements reach the long-cycle threshold $\ceil{n/7}$,  at which point Lemma~\ref{lem:long-completion} completes the cycle to a Hamiltonian cycle.

Our main result is the following.  

\begin{theorem}\label{thm:main}
Every 7-tough $(P_4\cup P_1)$-free graph on at least three vertices is Hamiltonian.
\end{theorem}

Section~2 records the auxiliary results and the path-cover compression. Section~3 contains the asymmetric separation criterion, the cograph covering lemma, the long-cycle completion lemma, and the cograph-core enlargement argument. The main theorem is proved in Section~4.

\section{Preliminaries}

All graphs are finite and simple. For a graph $G$, we write $V(G)$ and $E(G)$ for its vertex and edge sets, $\delta(G)$ for its minimum degree, $\alpha(G)$ for its independence number, and $\kappa(G)$ for its vertex-connectivity. For $v\in V(G)$ and $X\subseteq V(G)$, let
\[
N_G(v,X)=N_G(v)\cap X,\qquad d_G(v,X)=|N_G(v,X)|.
\]
We abbreviate $N_G(v,V(G))$ and $d_G(v,V(G))$ to $N_G(v)$ and $d_G(v)$. For a set $A\subseteq V(G)$, define its open neighborhood by
\[
N_G(A)=\left(\bigcup_{a\in A}N_G(a)\right)\setminus A
      =\{v\in V(G)\setminus A:N_G(v)\cap A\ne\varnothing\}.
\]
For $X\subseteq V(G)$, the subgraph induced by $X$ is denoted by $G[X]$.

For integers $p,q$, write $[p,q]=\{i\in\mathbb Z:p\le i\le q\}$; this set is empty when $p>q$. A set $X\subseteq V(J)$ is a cutset of a graph $J$ if $c(J-X)\ge2$. Under this convention, the empty set is a cutset when $J$ is disconnected. A cutset is minimal if none of its proper subsets is a cutset. If $J$ is connected and non-complete, a minimum vertex cut is a cutset of cardinality $\kappa(J)$.

Let $S\subseteq V(G)$, and let $D_1,\ldots,D_\ell$ be the components of $G-S$. Following Shan~\cite{Shan2026}, we say that $G$ is $t$-tough with respect to $S$ if
\[
|W|\ge t\,c(G-W)
\]
for every cutset $W$ of $G$ such that $V(D_i)\setminus W\ne\varnothing$ for all $i\in[1,\ell]$. Global $t$-toughness plainly implies $t$-toughness with respect to every set $S$.

A linear forest is a graph whose components are paths; isolated vertices are allowed unless stated otherwise. Following~\cite[Definition 2.10]{Shan2026}, a path-cover $Q$ of a graph $H$ is a union of pairwise vertex-disjoint paths such that $V(H)\subseteq V(Q)$. Thus a path-cover may use vertices outside $H$. It is spanning in $H$ when $V(Q)=V(H)$.

We use Shan's notion of an $S$-matched path-cover as follows. If $H\subseteq G-S$, an $S$-matched path-cover of $H$ is a path-cover $Q$ satisfying
\[
V(H)\subseteq V(Q)\subseteq V(H)\cup S,
\]
such that both endvertices of every component of $Q$ lie in $S$. It is an $S$-matched basic path-cover if no two $S$-vertices are adjacent in $Q$ and $Q-S$ is a basic path-cover of $H$.

Define
\[
\sigma_3(G)=
\begin{cases}
\min\{d_G(x)+d_G(y)+d_G(z):\{x,y,z\}\text{ is independent}\},&\alpha(G)\ge3,\\
\infty,&\alpha(G)\le2.
\end{cases}
\]

For a non-complete graph $H$, its scattering number is
\[
s(H)=\max\{c(H-U)-|U|:U\subseteq V(H),\ c(H-U)\ge2\}.
\]
For a complete graph we set $s(H)=-\infty$.

For a non-complete graph $H$, a set $W\subseteq V(H)$ is a \emph{scattering set} if
\[
c(H-W)-|W|=s(H)\qquad\text{and}\qquad c(H-W)\ge2.
\]
A scattering set is \emph{maximal} if it is maximal under inclusion among the scattering sets of $H$. If $X$ is a cutset of a graph, a vertex $z\in X$ is called a \emph{minimal element of $X$} if $z$ belongs to some minimal cutset contained in $X$.

To recall the part of Shan's definition used below, let $H$ be the union of some components of $G-S$, and put $W=\varnothing$ when $s(H)\le0$ and otherwise let $W$ be a maximal scattering set of $H$. A spanning path-cover $Q$ of $H$ with components $R_1,\ldots,R_k$ is a \emph{basic path-cover} if $k=\max\{1,s(H)\}$, the vertex set $V(R_1)$ consists of all vertices of $W$ together with all vertices of $|W|+1$ components of $H-W$, and, for every $i\in[2,k]$, the induced graph $H[V(R_i)]$ is a component of $H-W$.
We use repeatedly the following elementary consequences of toughness. If $G$ is a non-complete $t$-tough graph, then every cutset has order at least $2t$. Hence
\[
\kappa(G)\ge2t,\qquad \delta(G)\ge2t,\qquad |V(G)|\ge2t+1.
\]
Moreover, if $I$ is an independent set of order at least two, then toughness applied to $V(G)\setminus I$ gives
\[
|V(G)|-|I|\ge t|I|.
\]
Consequently, for an $n$-vertex non-complete $t$-tough graph,
\begin{equation}\label{eq:alpha}
\alpha(G)\le\frac{n}{t+1}.
\end{equation}

We first recall the Chv\'atal--Erd\H{o}s theorem.

\begin{lemma}[Chv\'atal--Erd\H{o}s~\cite{ChvatalErdos1972}]\label{lem:CE}
Every graph $J$ on at least three vertices satisfying $\kappa(J)\ge\alpha(J)$ is Hamiltonian.
\end{lemma}

The next lemma is the cycle-extension argument used throughout the paper.

\begin{lemma}[{Shan~\cite[Lemma~2.16]{Shan2026}}]\label{lem:cycle-extension}
Let $t>0$, let $G$ be a $t$-tough graph on $n$ vertices, and let $C$ be a non-Hamiltonian cycle of $G$. If $D$ is a connected subgraph of $G-V(C)$ such that
\[
|N_G(V(D))\cap V(C)|>\frac{n}{t+1}-1,
\]
then $G$ has a cycle $C'$ with
\[
V(C)\subsetneq V(C')\qquad\text{and}\qquad V(C')\cap V(D)\ne\varnothing.
\]
\end{lemma}

We shall use the following three auxiliary results of Shan.

The following fact is a direct consequence of Shan's structural lemmas for maximal scattering sets.

\begin{fact}[{ Shan~\cite[Lemmas~2.2(2), 2.5(1), and~2.5(3)]{Shan2026}}]\label{fact:scattering-inheritance}
Let $H$ be a connected $P_4$-free graph with $s(H)\ge0$, let $W$ be a maximal scattering set of $H$, and let $w\in W$ be a minimal element of $W$. Then $W\setminus\{w\}$ is a maximal scattering set of $H-w$. In particular,
\[
s(H-w)=s(H)+1.
\]
\end{fact}

\begin{lemma}[{Shan~\cite[Lemma~2.11 and its proof]{Shan2026}, one-component case}]\label{lem:shan-insertion}
Let $G$ be a $(P_4\cup P_1)$-free graph, let $S\subseteq V(G)$ be such that $G-S$ is $P_4$-free, and let $D$ be a component of $G-S$. Suppose that $s(D)\ge0$ and $D$ is not a complete bipartite graph. Let $W$ be a maximal scattering set of $D$, and let $z\in W$ be a minimal element of $W$. If $Q$ is a one-component $S$-matched basic path-cover of $D-z$ with $S$-endvertices $x$ and $y$, then $z$ can be inserted into $Q$ to obtain a one-component $S$-matched path-cover of $D$ with the same $S$-endvertices $x$ and $y$.
\end{lemma}

The first two assertions of the next lemma are due to
Shan~\cite[Lemma~2.12]{Shan2026}. The final endpoint assertion is
implicit in the construction used in Shan's proof. We state it explicitly
because it will be used in Lemma~2.6, and we include the details for
completeness.

\begin{lemma}[{Shan~\cite[Lemma~2.12]{Shan2026}, with an endpoint refinement}]\label{lem:shan-cover}
Let $G$ be a $(P_4\cup P_1)$-free graph and let $S\subseteq V(G)$ be such that $G-S$ is $P_4$-free. Then the following statements hold.
\begin{enumerate}[label=\textup{(\arabic*)}]
\item Suppose that $G$ is 4-tough with respect to $S$. If $s(G-S)\ge1$, then $G-S$ has an $S$-matched basic path-cover with exactly $s(G-S)$ components.
\item Suppose that $G$ is 4.5-tough with respect to $S$. If $s(G-S)\le0$, then $G-S$ has an $S$-matched path-cover with a single component.
\end{enumerate}
Moreover, in part~\textup{(2)}, if $|S|\ge2$ and $G-S$ is nonempty, the path-cover may be chosen to consist of an $(x,y)$-path $Q$ for two distinct vertices $x,y\in S$ such that
\[
V(Q)\cap S=\{x,y\},\qquad V(Q)\setminus S=V(G-S).
\]
Thus the only $S$-vertices on $Q$ are its two endvertices.
\end{lemma}

\begin{proof}[Proof of the endpoint refinement in Lemma~\ref{lem:shan-cover}]
It remains only to prove the additional endpoint assertion, which is the precise form needed in Lemma~\ref{lem:compression}. Put $H=G-S$. Since $s(H)\le0$, the graph $H$ is connected: otherwise the empty set would give $s(H)\ge c(H)\ge2$.

If $|V(H)|=1$, let $u$ be its unique vertex.  If $u$ had at most one neighbor in $S$, then deleting $N_G(u)\cap S$ would leave $u$ as one component and at least one vertex of $S$ in another component, contradicting the 4.5-toughness of $G$ with respect to $S$.  Hence $u$ has two distinct neighbors $x,y\in S$, and $xuy$ has the required form.

Assume now that $|V(H)|\ge2$.  There exist two independent edges $xu$ and $yv$ between $S$ and $V(H)$, with $x\ne y$ and $u\ne v$.  To see this, first observe that there is at least one edge between $S$ and $V(H)$; otherwise the empty set would be an admissible cutset. Fix such an edge $xu$. If there were no matching of size two, every edge between $S$ and $V(H)$ would meet $xu$. Moreover, there could not be both an edge $xu'$ with $u'\ne u$ and an edge $x'u$ with $x'\ne x$, because those two edges would be independent. Consequently, either $x$ or $u$ is incident with every edge between $S$ and $V(H)$. If $x$ is the common endpoint, then $H$ and the nonempty set $S\setminus\{x\}$ lie in different components of $G-x$; moreover, the unique component $H$ of $G-S$ is untouched, so $\{x\}$ is admissible with respect to $S$. If $u$ is the common endpoint, then $H-u\ne\varnothing$ and $H-u$ is anticomplete to $S$ in $G-u$; hence $\{u\}$ is again an admissible cutset with respect to $S$. Both alternatives contradict 4.5-toughness.

If $s(H)<0$, Jung's characterization~\cite{Jung1978} says that the $P_4$-free graph $H$ is Hamiltonian-connected.  Hence $H$ has a spanning $(u,v)$-path $P$, and $xuPvy$ is the desired path-cover.

Suppose next that $s(H)=0$ and $H$ is not complete bipartite.  Let $W$ be a maximal scattering set of $H$, and choose a minimal element $w\in W$.  By Fact~\ref{fact:scattering-inheritance}, $s(H-w)=1$.  Set $G'=G-w$.  We claim that $G'$ is 4-tough with respect to $S$.

Let $L$ be an admissible cutset of $G'$ with respect to $S$, and put $c=c(G'-L)\ge2$.  Write $H_1,\ldots,H_r$ for the components of $G'-S=H-w$.  Admissibility means that $V(H_i)\setminus L\ne\varnothing$ for every $i$.  In particular,
\[
H-(L\cup\{w\})=(H-w)-L\ne\varnothing.
\]
Since $H$ is the only component of $G-S$, this shows that $L\cup\{w\}$ is admissible for $G$ with respect to $S$.  Moreover,
\[
G-(L\cup\{w\})=G'-L,
\]
and hence
\[
c\bigl(G-(L\cup\{w\})\bigr)=c(G'-L)=c.
\]
The 4.5-toughness of $G$ with respect to $S$ now gives
\[
|L|+1\ge4.5c.
\]
Since $c\ge2$,
\[
|L|\ge4.5c-1\ge4c.
\]
Thus $G'$ is 4-tough with respect to $S$.  Part~\textup{(1)} gives an $S$-matched basic path-cover of $H-w$ with one component.  Its deletion of $S$ is a basic path-cover with one component, so an internal $S$-vertex would split it into at least two components.  The path therefore has no internal $S$-vertex.  Lemma~\ref{lem:shan-insertion} inserts $w$ while preserving the two $S$-endvertices and hence gives the required path through all vertices of $H$.

Finally, suppose that $s(H)=0$ and $H$ is complete bipartite, with bipartition $(A,B)$ and $|A|\le|B|$. Here $|B|\ge2$, since $K_{1,1}$ is complete and has scattering number $-\infty$ under our convention. Deleting $A$ leaves $|B|$ isolated vertices, so
\[
0=s(H)\ge |B|-|A|.
\]

Consequently $|A|=|B|$. We claim that each of $A$ and $B$ has a neighbor in $S$. If $A$ were anticomplete to $S$, then $B$ would be an admissible cutset: the graph $H-B$ consists of the $|A|$ isolated vertices of $A$, while the nonempty set $S$ lies in at least one further component. Thus
\[
c(G-B)\ge |A|+1,
\]
whereas $|B|=|A|<4.5(|A|+1)$, a contradiction. The same argument with $A$ and $B$ interchanged proves the claim.
 Put
\[
S_A=N_G(A)\cap S,\qquad S_B=N_G(B)\cap S.
\]
Both sets are nonempty. If no distinct pair $x\in S_A$ and $y\in S_B$ existed, then necessarily $S_A=S_B=\{s_0\}$ for some $s_0\in S$. Thus every edge between $S$ and $H$ would be incident with $s_0$, and deleting $s_0$ would be an admissible cutset of order one, a contradiction. Hence there are distinct vertices $x\in S_A$ and $y\in S_B$; choose $u\in A$ and $v\in B$ with $xu,yv\in E(G)$. The balanced complete bipartite graph $H$ has a spanning $(u,v)$-path $P$, so $xuPvy$ is the desired path-cover. This completes the proof of the endpoint refinement in Lemma~\ref{lem:shan-cover}.
\end{proof}

The following lemma records, in the form needed below, the
path-cover compression used in the proof of
Shan~\cite[Theorem~1.2, Case~2]{Shan2026}. The underlying
path-cover is supplied by
Shan~\cite[Lemma~2.12]{Shan2026}. Since the linear-forest
formulation, the bound $m\le n/8$, and the simultaneous expansion
statement are not stated there as a single result, we give a
complete proof.
\begin{lemma}\label{lem:compression}
Let $G$ be a non-complete 7-tough $(P_4\cup P_1)$-free graph on $n$ vertices. Let $S\subseteq V(G)$ satisfy $|S|\ge2$, and suppose that $G-S$ is nonempty and $P_4$-free. Put
\[
m=\max\{1,s(G-S)\}.
\]
Then $m\le n/8$. Moreover, there are
\begin{itemize}
\item a linear forest $F$ on a subset of $S$, with exactly $m$ edges and no isolated vertices, and
\item for every edge $e=xy\in E(F)$, an $(x,y)$-path $Q_e$ in $G$,
\end{itemize}
such that the interiors of the paths $Q_e$ are pairwise disjoint, lie in $G-S$, and partition $V(G-S)$. If $R=G[S]$ and $H=R+E(F)$, with the edges of $F$ regarded as marked edges, then every cycle $C_H$ of $H$ containing all marked edges expands to a cycle $C$ of $G$ satisfying
\[
|V(C)|=|V(C_H)|+|V(G-S)|.
\]
\end{lemma}

\begin{proof}
Global 7-toughness implies 7-toughness, and hence both 4-toughness and 4.5-toughness, with respect to $S$.

Suppose first that $s(G-S)=m\ge1$. By Lemma~\ref{lem:shan-cover}(1), $G-S$ has an $S$-matched basic path-cover $Q$ with $c(Q)=m$. By the definition of a basic path-cover, $Q-S$ also has exactly $m$ nonempty path components; denote them by
\[
P_1,\ldots,P_m.
\]
Their vertex sets are pairwise disjoint and partition $V(G-S)$. Every component of $Q$ is a path with both endvertices in $S$, and no two $S$-vertices are adjacent in $Q$. Consequently, each maximal subpath $P_i$ of $Q-S$ is flanked in its component of $Q$ by two distinct vertices $x_i,y_i\in S$. Put
\[
e_i=x_i y_i,\qquad Q_{e_i}=x_iP_i y_i.
\]
Suppress every $P_i$ to the edge $e_i$. On each component of $Q$, this operation produces a path on its $S$-vertices. After isolated $S$-vertices are discarded, the union of these paths is a linear forest $F$ with edge set $\{e_1,\ldots,e_m\}$. Distinct $P_i$ give distinct edges: two such subpaths in one component of $Q$ cannot have the same two flanking vertices without repeating an $S$-vertex on a simple path, while distinct components of $Q$ are vertex-disjoint.

We next assume that $s(G-S)\le0$, so that $m=1$.  By Lemma~\ref{lem:shan-cover}(2) and its final endpoint assertion, there are distinct vertices $x,y\in S$ and an $(x,y)$-path $Q_{xy}$ such that
\[
V(Q_{xy})\cap S=\{x,y\},\qquad V(Q_{xy})\setminus S=V(G-S).
\]
Let $F$ consist of the single marked edge $xy$ and assign to it the path $Q_{xy}$.

We now prove the bound on $m$. If $m=1$, then $n\ge15$, because every non-complete 7-tough graph has at least $15$ vertices. Hence $m\le n/8$. Suppose that $m=s(G-S)\ge1$, and choose $U\subseteq V(G-S)$ such that
\[
c(G-S-U)-|U|=m.
\]
Then $c(G-S-U)=m+|U|\ge2$, so $S\cup U$ is a cutset of $G$. Toughness gives
\[
|S|+|U|\ge7(m+|U|).
\]
Moreover, $G-S-U$ has $m+|U|$ nonempty components, and therefore
\[
|V(G-S)|\ge |U|+c(G-S-U)=m+2|U|.
\]
Consequently,
\[
n+|U|=|S|+|U|+|V(G-S)|\ge8m+9|U|,
\]
and in particular $n\ge8m$.

It remains to justify the expansion assertion. The edges of $F$ remain distinguished as marked edges even if some already belong to $R=G[S]$. Let a component of $F$ have its vertices in path order
\[
x_0,x_1,\ldots,x_k,
\]
and write $e_i=x_{i-1}x_i$. In a cycle containing every marked edge, the path $x_0x_1\cdots x_k$ occurs as a contiguous segment: at each internal vertex $x_i$, the two incident marked edges already occupy the two cycle edges. Orient each assigned path $Q_{e_i}$ from $x_{i-1}$ to $x_i$ and replace this marked segment by
\[
Q_{e_1}Q_{e_2}\cdots Q_{e_k},
\]
writing each common endvertex only once. Different components of $F$ have disjoint vertex sets, and the interiors of all assigned paths are pairwise disjoint, lie in $G-S$, and together partition $V(G-S)$. Hence the replacements can be made simultaneously and produce a simple cycle $C$ of $G$. The added vertices are precisely the vertices of $G-S$, so
\[
|V(C)|=|V(C_H)|+|V(G-S)|.
\]
\end{proof}
We shall use Shan's covering-cycle lemma for minimal cutsets.

\begin{lemma}[{Shan~\cite[Lemma~2.15(2)]{Shan2026}}]\label{lem:shan-cycle}
Let $G$ be a 4.5-tough $(P_4\cup P_1)$-free graph and let $X$ be a minimal cutset of $G$. Then $G$ has a cycle containing every vertex of $G-X$.
\end{lemma}

The following ordered insertion lemma follows from Lemmas~\ref{lem:cycle-extension} and~\ref{lem:shan-cycle}.

\begin{lemma}[{parameterized form of
		Shan~\cite[Lemma~2.17]{Shan2026}}]\label{lem:ordered-insertion}
Let $t\ge4.5$, let $G$ be a $t$-tough $(P_4\cup P_1)$-free graph on $n$ vertices, and let $X$ be a cutset. Suppose that $X_0=\{z_1,\ldots,z_q\}\subseteq X$ is ordered so that
\[
d_G\bigl(z_i,(V(G)\setminus X)\cup\{z_1,\ldots,z_{i-1}\}\bigr)>\frac{n}{t+1}-1
\]
for every $i$. Then $G$ has a cycle containing every vertex of $(V(G)\setminus X)\cup X_0$.
\end{lemma}

We use the following specialization of the theorem of Hu, Tian, and Wei on cycles through a prescribed linear forest.

\begin{lemma}[{Hu--Tian--Wei~\cite[Theorem~3]{HuTianWei2001}}]\label{lem:HTW}
Let $q\ge0$. Let $J$ be a $(q+2)$-connected graph, and let $F$ be a linear forest in $J$ having $q$ edges and no isolated components. Then $J$ has a cycle containing every edge of $F$ and having order at least
\[
\min\set{|V(J)|,\frac23\sigma_3(J)-q}.
\]
The empty forest is included when $q=0$.
\end{lemma}

\section{Structural lemmas}

Throughout this section, $G$ is a 7-tough $(P_4\cup P_1)$-free graph on $n$ vertices, and
\[
a=\frac n8.
\]
The insertion threshold in Lemma~\ref{lem:ordered-insertion} is $a-1$.

\subsection{An asymmetric separation criterion}

\begin{lemma}\label{lem:asymmetric}
Let $X$ be a cutset of $G$. Suppose that $G-X$ has a component $D$ such that either
\begin{equation}\label{eq:side1}
|D|>2a-2,\qquad \sum_{D'\ne D}|D'|>a-1,
\end{equation}
or
\begin{equation}\label{eq:side2}
|D|>a-1,\qquad \sum_{D'\ne D}|D'|>2a-2.
\end{equation}
Then $G$ is Hamiltonian.
\end{lemma}

\begin{proof}
We first normalize the cutset. At every stage, keep track of the component containing the original component $D$ and call it the distinguished component; all other components form the second side. If a vertex $x$ of the current cutset has no neighbor in the distinguished component, remove $x$ from the cutset. The distinguished component is unchanged, while $x$ joins, merges, or creates components on the second side. If $x$ has a neighbor in the distinguished component but no neighbor in any other component, remove $x$ from the cutset; it is then absorbed into the distinguished component. In the first operation the order of the distinguished side is unchanged and the total order of the second side does not decrease; in the second operation the distinguished side grows and the second side is unchanged. Thus the relevant strict inequalities in \eqref{eq:side1} or \eqref{eq:side2} are preserved throughout. Repeating the operations yields a cutset, still denoted by $X$, such that every vertex of $X$ has a neighbor in the distinguished component and in at least one other component.

Let the components of $G-X$ be $D_1,\ldots,D_\ell$, where $D_1$ is the distinguished component. If $\ell\ge3$, fix $x\in X$. If $x$ is complete to $D_1$, then
\[
d_G(x,G-X)\ge|D_1|>a-1.
\]
Suppose that $x$ is not complete to $D_1$. Since $D_1$ is connected and $x$ has a neighbor in $D_1$, there is an edge $uv\in E(D_1)$ such that $xu\in E(G)$ and $xv\notin E(G)$. Choose a neighbor $w\in D_2$ of $x$. Then
\[
w,x,u,v
\]
induces a $P_4$. Every vertex of $D_3\cup\cdots\cup D_\ell$ must therefore be adjacent to $x$. Repeating the same argument with a vertex of $D_3$ in place of $w$ shows that $x$ is complete to $D_2$ as well. Thus $x$ is complete to the whole second side, and again
\[
d_G(x,G-X)>a-1.
\]
Lemma~\ref{lem:ordered-insertion}, used in any order on $X$, now gives a Hamiltonian cycle.

It remains to consider $\ell=2$. After relabeling the two components if necessary, we may assume
\begin{equation}\label{eq:D12}
|D_1|>2a-2,\qquad |D_2|>a-1.
\end{equation}
Every vertex of $X$ has a neighbor in both components. We claim that $X$ is a minimal cutset. Indeed, let $X'\subsetneq X$ and choose $x\in X\setminus X'$. In $G-X'$, the vertex $x$ joins the connected graphs $D_1$ and $D_2$; every other vertex of $X\setminus X'$ has a neighbor in each of them and hence lies in the same component. Thus $G-X'$ is connected, proving the claim.

Put
\[
X_0=\{x\in X:d_G(x,D_1\cup D_2)\le a-1\}.
\]
If $X_0=\varnothing$, Lemma~\ref{lem:ordered-insertion} applies directly. Assume that $X_0\ne\varnothing$.

We need three elementary observations. First, if $x\in X_0$, $i\in\{1,2\}$, and $z\in N_G(x,D_i)$, then
\begin{equation}\label{eq:complete-rest}
z\text{ is complete to }D_i\setminus N_G(x,D_i).
\end{equation}
Indeed, in the other component $D_{3-i}$ the vertex $x$ has both a neighbor and a nonneighbor, by \eqref{eq:D12} and the definition of $X_0$. Since $D_{3-i}$ is connected, choose an edge $pq$ with $xp\in E(G)$ and $xq\notin E(G)$. The vertices $z,x,p,q$ induce a $P_4$. Any vertex in $D_i\setminus N_G(x,D_i)$ that is not adjacent to $z$ would be isolated from this path.

Second, if $x,y\in X_0$ are nonadjacent, then
\begin{equation}\label{eq:comparable}
N_G(x,D_1)\text{ and }N_G(y,D_1)\text{ are comparable by inclusion}.
\end{equation}
Suppose first that the neighborhoods of $x$ and $y$ in $D_2$ are incomparable. Choose
\[
p\in N_G(x,D_2)\setminus N_G(y,D_2),\qquad q\in N_G(y,D_2)\setminus N_G(x,D_2).
\]
By \eqref{eq:complete-rest}, $pq\in E(G)$, and $x,p,q,y$ induces a $P_4$. Hence every vertex of $D_1$ is adjacent to $x$ or $y$. This gives
\[
d_G(x,D_1)+d_G(y,D_1)\ge|D_1|>2a-2,
\]
so one of $x,y$ has more than $a-1$ neighbors in $D_1$, a contradiction.

Thus the neighborhoods in $D_2$ are comparable; assume
\[
N_G(y,D_2)\subseteq N_G(x,D_2).
\]
Since $d_G(x,D_2)\le a-1<|D_2|$, there is a common nonneighbor of $x$ and $y$ in $D_2$. If their neighborhoods in $D_1$ were incomparable, the same construction in $D_1$, together with this common nonneighbor, would yield an induced $P_4\cup P_1$. This proves \eqref{eq:comparable}.

Third, if $x,y\in X_0$ are adjacent and their neighborhoods in $D_1$ are incomparable, then
\begin{equation}\label{eq:coverD2}
D_2\subseteq N_G(x,D_2)\cup N_G(y,D_2).
\end{equation}
Choose
\[
p\in N_G(x,D_1)\setminus N_G(y,D_1).
\]
Since
\[
|N_G(x,D_1)\cup N_G(y,D_1)|\le2a-2<|D_1|,
\]
there is a vertex $q\in D_1$ adjacent to neither $x$ nor $y$. By \eqref{eq:complete-rest}, $pq\in E(G)$. Thus $y,x,p,q$ is an induced $P_4$, and every vertex of $D_2$ must be adjacent to $x$ or $y$.

Consider the family
\[
\mathcal N=\{N_G(x,D_1):x\in X_0\}.
\]
Let $M_1,\ldots,M_r$ be its distinct inclusion-maximal members, and choose $x_i\in X_0$ with $M_i=N_G(x_i,D_1)$. Distinct maximal members are incomparable. By \eqref{eq:comparable}, the vertices $x_1,\ldots,x_r$ are pairwise adjacent. By \eqref{eq:coverD2},
\[
N_G(x_i,D_2)\cup N_G(x_j,D_2)=D_2\qquad(i\ne j).
\]
Consequently, the sets
\[
D_2\setminus N_G(x_i,D_2),\qquad i=1,\ldots,r,
\]
are pairwise disjoint, and hence
\begin{equation}\label{eq:sumD2}
\sum_{i=1}^r d_G(x_i,D_2)\ge(r-1)|D_2|.
\end{equation}
Set
\[
U=\bigcup_{x\in X_0}N_G(x,D_1)=\bigcup_{i=1}^r M_i.
\]
If $r=1$, then $x_1$ has at least one neighbor in $D_2$, and therefore
\[
|U|=|M_1|\le a-2<a-1.
\]
If $r\ge2$, then \eqref{eq:sumD2} and \eqref{eq:D12} give
\[
\begin{aligned}
|U|&\le\sum_{i=1}^r|M_i|\\
&\le r(a-1)-\sum_{i=1}^r d_G(x_i,D_2)\\
&\le r(a-1)-(r-1)|D_2|<a-1.
\end{aligned}
\]
Every vertex of $U$ is complete to $D_1\setminus U$ by \eqref{eq:complete-rest}, and
\[
|D_1\setminus U|>|D_1|-(a-1)>a-1.
\]
Now put
\[
X^*=(X\setminus X_0)\cup U.
\]
The nonempty set $D_1\setminus U$ has no neighbor in $X_0$, and it is anticomplete to $D_2$. Hence $X^*$ is a cutset. Order the vertices of $X^*$ by listing $U$ first and $X\setminus X_0$ second. For $u\in U$,
\[
d_G(u,V(G)\setminus X^*)\ge|D_1\setminus U|>a-1.
\]
For $x\in X\setminus X_0$, after all vertices of $U$ have been inserted,
\[
d_G\bigl(x,(V(G)\setminus X^*)\cup U\bigr)\ge d_G(x,D_1\cup D_2)>a-1.
\]
Lemma~\ref{lem:ordered-insertion} gives a Hamiltonian cycle.
\end{proof}

\subsection{Cycles covering prescribed cograph sets}

To apply Lemma \ref{lem:asymmetric}, we need a way to construct an initial cycle through the low-degree core; the following lemma provides this.

\begin{lemma}\label{lem:cograph-cover}
Let $G$ be a 4.5-tough $(P_4\cup P_1)$-free graph on at least three vertices, and let $Z\subseteq V(G)$ induce a $P_4$-free graph. Then $G$ has a cycle containing every vertex of $Z$.
\end{lemma}

\begin{proof}
If $G$ is complete, the assertion is clear. A 4.5-tough non-complete graph is at least 9-connected. In particular it is 2-connected, and any two prescribed vertices of a 2-connected graph lie on a common cycle. Thus the assertion is immediate when $|Z|\le2$. We assume $|Z|\ge3$.

Suppose first that $G[Z]$ is disconnected. Then $V(G)\setminus Z$ is a cutset. Choose a minimal cutset $X\subseteq V(G)\setminus Z$. By Lemma~\ref{lem:shan-cycle}, $G$ has a cycle containing every vertex of $G-X$, and hence every vertex of $Z$.

We may therefore assume that $G[Z]$ is connected. If $G[Z]$ is Hamiltonian, there is nothing to prove. Suppose that it is not Hamiltonian. By Dirac's theorem~\cite{Dirac1952},
\[
\delta(G[Z])<\frac{|Z|}{2}.
\]
The graph $G[Z]$ is non-complete, so it has a minimum vertex cut $U$ satisfying
\begin{equation}
|U|\le\delta(G[Z])<\frac{|Z|}{2}.
\end{equation}
Put $R=Z\setminus U$ and $r=|R|$. Then $G[R]$ is disconnected and
\begin{equation}\label{eq:rU}
r>|U|.
\end{equation}
The minimum vertex cut $U$ is inclusion-minimal. Moreover, every $u\in U$ has a neighbor in every component of $G[R]$: if $u$ had no neighbor in one such component, then $U\setminus\{u\}$ would still disconnect $G[Z]$. We now verify directly that every vertex of $U$ is complete to $R$.  Fix $u\in U$ and a component $K$ of $G[R]$.  If $u$ were not complete to $K$, then, since $u$ has a neighbor in $K$ and $K$ is connected, there would be an edge $ab\in E(K)$ with $ua\in E(G)$ and $ub\notin E(G)$.  Choose a different component $K'$ of $G[R]$ and a neighbor $c\in V(K')$ of $u$.  The vertices $c,u,a,b$ would induce a $P_4$, contradicting that $G[Z]$ is $P_4$-free.  Hence $u$ is complete to every component of $G[R]$, and therefore $U$ is complete to $R$.

The set
\[
X=(V(G)\setminus Z)\cup U
\]
is a cutset because $G-X=G[R]$ is disconnected. Choose a minimal cutset $X'\subseteq X$. Lemma~\ref{lem:shan-cycle} gives a cycle $C$ containing every vertex of $G-X'$. In particular,
\[
R\subseteq V(C).
\]
Let
\[
U_{\rm in}=U\cap V(C),\qquad U_{\rm out}=U\setminus V(C).
\]
List the vertices of $R$ in their cyclic order on $C$ as
\[
r_1,r_2,\ldots,r_{r}.
\]
For each $i$, let $Q_i$ be the $r_i$--$r_{i+1}$ arc of $C$ whose interior contains no vertex of $R$, where indices are taken modulo $r$. At most $|U_{\rm in}|$ of these $r$ arcs contain a vertex of $U_{\rm in}$ in their interiors. By \eqref{eq:rU}, the number of remaining arcs is at least
\[
r-|U_{\rm in}|>|U|-|U_{\rm in}|=|U_{\rm out}|.
\]

Assign to every $x\in U_{\rm out}$ a distinct arc $Q_i$ whose interior contains no vertex of $Z$. Since $x$ is adjacent to every vertex of $R$, replace the assigned arc $Q_i$ by the two-edge path
\[
r_i x r_{i+1}.
\]
For each $i$, let $L_i$ denote either the original arc $Q_i$ or its replacement. The interiors of $L_1,\ldots,L_r$ are pairwise disjoint, and consecutive members meet only at their common vertex in $R$. Their union, taken in the cyclic order
\[
r_1,L_1,r_2,L_2,\ldots,r_r,L_r,r_1,
\]
is therefore a connected 2-regular graph, and hence a simple cycle. No replaced arc contains a vertex of $U_{\rm in}$ in its interior, while the inserted vertices are precisely the distinct vertices of $U_{\rm out}$. The new cycle consequently contains
\[
R\cup U_{\rm in}\cup U_{\rm out}=Z.
\]

\end{proof}

The following observation is an immediate consequence of the
standard inequality
$\rho(H)\ge \max\{1,s(H)\}$ \cite[p.~21]{GiakoumakisRousselThuillier1997}, where $\rho(H)$ denotes the minimum
number of pairwise vertex-disjoint paths covering $V(H)$. For completeness, we include a short proof.

\begin{lemma}\label{lem:scattering-cycle}
Let $H$ be an induced subgraph of an ambient graph, with $s(H)\ge1$, and let $C$ be a cycle of the ambient graph such that $V(H)\subseteq V(C)$. Then
\[
|V(C)\setminus V(H)|\ge s(H).
\]
Consequently,
\[
|V(C)|\ge|V(H)|+s(H).
\]
\end{lemma}

\begin{proof}
Indeed, put $r=|V(C)\setminus V(H)|$. Since $s(H)\ge1$, we have
$r\ge1$. Deleting these $r$ vertices from $C$ gives a path-cover
of $H$ with at most $r$ components. Hence
$s(H)\le\rho(H)\le r$.
\end{proof}
\subsection{Completion of a sufficiently long cycle}
Once a sufficiently large cycle through the cograph core is available, the following completion lemma forces it to be Hamiltonian.
\begin{lemma}\label{lem:long-completion}
Let $G$ be a 7-tough $(P_4\cup P_1)$-free graph on $n$ vertices, and put
\[
\lambda=\ceil{\frac n7}.
\]
If $G$ has a cycle $C$ such that
\[
|V(C)|\ge\lambda
\]
and every vertex outside $C$ has degree at least $n/4$, then $G$ is Hamiltonian.
\end{lemma}

\begin{proof}
Put $a=n/8$. If $G$ is complete, the conclusion is immediate. Thus $G$ is connected. Among all cycles satisfying the hypotheses, choose one of maximum order, still denoted by $C$. Suppose that $C$ is not Hamiltonian, and let $D$ be a component of $G-V(C)$. Since $G$ is connected, $D$ has a neighbor on $C$. Put
\[
X=N_G(V(D))\cap V(C),\qquad r=|X|,\qquad d=|D|.
\]
We claim that
\begin{equation}\label{eq:rbound}
r\le a-1.
\end{equation}
Indeed, if $r>a-1=n/(7+1)-1$, then Lemma~\ref{lem:cycle-extension} gives a cycle $C'$ satisfying $V(C)\subsetneq V(C')$. Consequently $|V(C')|>|V(C)|\ge\lambda$, while
\[
V(G)\setminus V(C')\subseteq V(G)\setminus V(C).
\]
Every vertex outside $C'$ therefore still has degree at least $n/4$. Thus $C'$ belongs to the same family of cycles over which $C$ was chosen maximal, contradicting the maximality of $C$. This proves \eqref{eq:rbound}.
The same maximality argument applies to every cycle extension constructed in the proof of Lemma~\ref{lem:cycle-extension}: each such extension would still have order at least $\lambda$, and every vertex outside it would still have degree at least $n/4$. Hence none of those extensions can occur. Following that proof, after orienting $C$ and putting
\[
X^+=\{x^+:x\in X\},
\]
we conclude that no two vertices of $X$ are consecutive on $C$, that $X^+$ is independent, and that $X^+$ is anticomplete to $D$. Consequently $X\cap X^+=\varnothing$ and $|X^+|=|X|=r$.

Let
\[
R=n-r-d
\]
be the total order of the components of $G-X$ other than $D$.

Suppose first that $d\le2a-2$. Every vertex of $D$ has degree at least $n/4=2a$, and all its neighbors outside $D$ lie in $X$. Hence, by \eqref{eq:rbound},
\[
d\ge2a-r+1\ge a+2>a-1.
\]
Moreover,
\[
R=n-r-d\ge8a-(a-1)-(2a-2)=5a+3>2a-2.
\]
Lemma~\ref{lem:asymmetric}, applied to the cutset $X$, gives a Hamiltonian cycle.

It remains to suppose that $d>2a-2$. The vertices of $C-X$ lie in components of $G-X$ other than $D$, so
\begin{equation}\label{eq:Rcycle}
R\ge|C|-r\ge\lambda-r.
\end{equation}
Here $r\ge1$ by the connectedness observation above. Let
\[
Q=V(G)\setminus\bigl(V(D)\cup X^+\bigr).
\]
The graph $G-Q$ consists of the connected graph $D$ and the $r$ isolated vertices of $X^+$. Thus $Q$ is a cutset. Using $|X^+|=r$, we have
\[
|Q|=n-d-r=R.
\]
Therefore 7-toughness gives
\begin{equation}\label{eq:Rrough}
R=|Q|\ge7(r+1).
\end{equation}
If $R\le a-1$, then \eqref{eq:Rcycle} implies
\[
r\ge\lambda-a+1.
\]
Using $\lambda\ge n/7=8a/7$ in \eqref{eq:Rrough}, we obtain
\[
R\ge7(r+1)\ge7(\lambda-a+2)\ge a+14>a-1,
\]
a contradiction. Hence $R>a-1$, and Lemma~\ref{lem:asymmetric} again yields a Hamiltonian cycle.
\end{proof}

\subsection{Separation and enlargement of a cograph core}
When no large cycle through the core exists, the failure must be caused by a separation of the auxiliary graph; the next lemma extracts a cograph obstruction from such a separation.
\begin{lemma}\label{lem:separation-obstruction}
Let $G$ be a $(P_4\cup P_1)$-free graph on $n$ vertices. Let $A,B,T$ be pairwise disjoint vertex sets such that $A$ is anticomplete to $B$. Suppose that
\[
G[T]\text{ is }P_4\text{-free},\qquad d_G(x)<\frac n4\quad(x\in T),
\]
and
\[
|B|>\frac n2,\qquad |A|+|B|>\frac{3n}{4}-1.
\]
Then $G[A\cup T]$ is $P_4$-free.
\end{lemma}

\begin{proof}
Suppose that $P$ is an induced $P_4$ in $G[A\cup T]$. Since $G[T]$ is $P_4$-free, the path contains a vertex of $A$.

It cannot contain at least two vertices of $A$. In that case it contains at most two vertices of $T$. Their degree sum is smaller than $n/2$, while $|B|>n/2$. Hence some vertex of $B$ is nonadjacent to all $T$-vertices of $P$. It is also nonadjacent to the $A$-vertices of $P$, and therefore is isolated from $P$, a contradiction.

Thus
\[
V(P)=\{a\}\cup X,\qquad a\in A,\qquad X=\{x_1,x_2,x_3\}\subseteq T.
\]
Every vertex of $B$ has a neighbor in $X$, since $A$ is anticomplete to $B$. Therefore
\begin{equation}
\sum_{x\in X}d_G(x,B)\ge|B|.
\end{equation}
We claim that
\[
A_0:=A\setminus N_G(X)\ne\varnothing.
\]
Assume otherwise. If $a$ is an endvertex of $P$, the three vertices of $X$ induce two edges. Hence
\[
\sum_{x\in X}d_G(x)\ge|A|+|B|+4>\frac{3n}{4},
\]
contrary to $X\subseteq T$. If $a$ is internal on $P$, then $a$ has two neighbors in $X$, while $G[X]$ has one edge. Thus
\[
\sum_{x\in X}d_G(x)\ge|A|+|B|+3>\frac{3n}{4},
\]
the same contradiction. Hence $A_0\ne\varnothing$.

Put
\[
B_2=\{b\in B:|N_G(b)\cap X|\ge2\},\qquad B_1=B\setminus B_2.
\]
Every vertex of $B_1$ has exactly one neighbor in $X$, and
\[
|B|+|B_2|\le\sum_{x\in X}d_G(x,B)<\frac{3n}{4}.
\]
Consequently,
\[
|B_1|=|B|-|B_2|>2|B|-\frac{3n}{4}>\frac n4.
\]
Suppose first that $a$ is an endvertex of $P$, say $P=ax_1x_2x_3$. If a vertex of $B_1$ has unique neighbor $x_1$ or $x_3$, it forms an induced $P_4$ with three vertices of $P$, while any vertex of $A_0$ is isolated from that path. Hence every vertex of $B_1$ has unique neighbor $x_2$, and
\[
d_G(x_2)\ge|B_1|>\frac n4,
\]
a contradiction.

Finally, suppose that $a$ is internal, say $P=x_1ax_2x_3$. We first claim that $A_0\subseteq N_G(a)$. Otherwise, take $u\in A_0\setminus N_G(a)$ and $b\in B_1$. According as the unique neighbor of $b$ in $X$ is $x_1,x_2$, or $x_3$, one of
\[
bx_1ax_2,\qquad bx_2ax_1,\qquad bx_3x_2a
\]
is an induced $P_4$ from which $u$ is isolated. Now take $u\in A_0$. Then $ua\in E(G)$, and $uax_2x_3$ is an induced $P_4$. Every vertex of $B$ must therefore be adjacent to $x_2$ or $x_3$. Hence
\[
|B|\le d_G(x_2)+d_G(x_3)<\frac n2,
\]
a contradiction.
\end{proof}

For the remainder of this subsection assume that $n>112$, and put
\[
h=\ceil{\frac n4},\qquad \lambda=\ceil{\frac n7},\qquad b=h-\lambda,\qquad \mu=\frac n2-\lambda-\frac n8.
\]
The following elementary inequalities will be used repeatedly:
\begin{equation}\label{eq:constants}
2b\ge\lambda,\qquad \lambda+1<\mu,\qquad n-\lambda-\frac n8>\frac n2,\qquad n-\lambda-1>\frac{3n}{4}-1.
\end{equation}
Indeed, $h\ge n/4$ and $\lambda\le n/7+1$, so
\[
2h-3\lambda\ge\frac n{14}-3>0,
\]
which gives $2b=2h-2\lambda\ge\lambda$, and
\[
\mu-(\lambda+1)\ge\frac{5n}{56}-3>0.
\]
Moreover,
\[
n-\lambda-\frac n8-\frac n2=\frac{3n}{8}-\lambda\ge\frac{13n}{56}-1>0,
\]
and
\[
n-\lambda-1-\left(\frac{3n}{4}-1\right)=\frac n4-\lambda\ge\frac{3n}{28}-1>0.
\]

The obstruction lemma can be turned into an amplification mechanism: when a separation occurs, the cograph core can be enlarged. The following lemma formalizes one enlargement step.

\begin{lemma}\label{lem:enlargement}
Suppose that
\[
V(G)=P\mathbin{\dot\cup}T\mathbin{\dot\cup}B\mathbin{\dot\cup}W
\]
satisfies the following conditions:
\begin{enumerate}[label=\textup{(\roman*)}]
\item every vertex of $P\cup B\cup W$ has degree at least $n/4$, while every vertex of $T$ has degree smaller than $n/4$;
\item $P$ is anticomplete to $B$;
\item $Z:=P\cup T$ is nonempty and $G[Z]$ is $P_4$-free;
\item $|W|<\mu$;
\item no cycle of order at least $\lambda$ contains every vertex of $Z$.
\end{enumerate}
Then either $G$ is Hamiltonian, or there is a new partition
\[
V(G)=P'\mathbin{\dot\cup}T\mathbin{\dot\cup}B'\mathbin{\dot\cup}W'
\]
satisfying conditions \textup{(i)}--\textup{(iii)} and
\[
|P'\cup T|\ge|P\cup T|+b.
\]
If $W=\varnothing$, the new partition also satisfies $|W'|<\mu$.
\end{lemma}

\begin{proof}
If $G$ is complete, the first alternative holds. We may therefore assume that $G$ is non-complete. Write $z=|Z|$ and define
\[
q=\begin{cases}
s(G[Z]),&\text{if }s(G[Z])\ge1,\\
1,&\text{if }s(G[Z])\le0.
\end{cases}
\]
Lemma~\ref{lem:cograph-cover} gives a cycle containing every vertex of $Z$. If $s(G[Z])\ge1$, Lemma~\ref{lem:scattering-cycle} shows that such a cycle has order at least $z+q$, and assumption \textup{(v)} gives $z+q\le\lambda-1$. If $s(G[Z])\le0$, then $q=1$; assumption \textup{(v)} gives $z<\lambda$, and since $z$ and $\lambda$ are integers, $z\le\lambda-1$, whence $z+q\le\lambda$. Thus, in both cases,
\begin{equation}\label{eq:zq}
z+q\le\lambda.
\end{equation}
Put $Y=V(G)\setminus Z=B\cup W$. Since $z<\lambda<n-2$, the set $Y$ has at least two vertices. Apply Lemma~\ref{lem:compression} with $S=Y$. We obtain a marked linear forest $F$ on $Y$ with $q$ edges, whose assigned path interiors partition $Z$; in particular,
\[
q\le a.
\]
Since $z<\lambda$ and $n>112$, we have
\[
|Y|=n-z>n-\lambda>a+2\ge q+2.
\]
Let
\[
R=G[Y],\qquad H=R+E(F).
\]
Every vertex of $Y$ has degree at least $h$, and hence
\begin{equation}\label{eq:dR}
d_R(v)\ge h-z\qquad(v\in Y).
\end{equation}

Suppose first that $R$ is $(q+2)$-connected. Then $H$ is also $(q+2)$-connected. If $\alpha(H)\le2$, then $\sigma_3(H)=\infty$, and Lemma~\ref{lem:HTW} gives a Hamiltonian cycle of $H$ containing every marked edge; its expansion is Hamiltonian in $G$. We may therefore assume that $\alpha(H)\ge3$. Then

\[
\sigma_3(H)\ge3(h-z).
\]
Lemma~\ref{lem:HTW} gives a cycle $C_H$ of $H$ through all marked edges and of order at least
\[
\min\{|Y|,2(h-z)-q\}.
\]
If $|Y|\le2(h-z)-q$, this lower bound forces $C_H$ to be Hamiltonian in $H$, and its expansion is Hamiltonian in $G$. Otherwise $2(h-z)-q<|Y|$, so either $C_H$ is Hamiltonian or its expansion has order at least
\[
2(h-z)-q+z=2h-z-q\ge2h-\lambda\ge\lambda,
\]
where \eqref{eq:zq} was used. The expanded cycle contains $Z$, so every vertex outside it has degree at least $n/4$. Lemma~\ref{lem:long-completion} gives a Hamiltonian cycle.

We may therefore assume that $R$ is not $(q+2)$-connected. Since $|Y|>q+2$, there is a set $U\subseteq V(R)$ such that
\begin{equation}\label{eq:Ucut}
|U|\le q+1\qquad\text{and}\qquad c(R-U)\ge2.
\end{equation}
Indeed, take $U=\varnothing$ if $R$ is disconnected, and otherwise take a vertex cut of order at most $q+1$. For every component $K$ of $R-U$, choose $v\in V(K)$. By \eqref{eq:dR}, \eqref{eq:Ucut}, and \eqref{eq:zq},
\begin{equation}\label{eq:Klower}
|K|\ge d_R(v)-|U|+1\ge h-z-q\ge b.
\end{equation}
Moreover,
\begin{equation}\label{eq:RminusU}
|V(R-U)|=n-z-|U|\ge n-z-q-1\ge n-\lambda-1.
\end{equation}
Apply Lemma~\ref{lem:asymmetric} to the cutset
\[
X:=Z\cup U.
\]
Indeed,
\[
G-X=G[Y]-U=R-U,
\]
so the components of $G-X$ are precisely the components of $R-U$. If every component of $R-U$ has order greater than $a-1$, Lemma~\ref{lem:asymmetric} yields a Hamiltonian cycle. Indeed, if there are at least three components, one component has order greater than $a-1$ and two others have total order greater than $2a-2$. If there are exactly two components, the larger has order greater than $2a-2$, since
\[
\frac{n-\lambda-1}{2}>2a-2.
\]
Thus some component $A$ of $R-U$ satisfies
\begin{equation}
|A|\le a-1.
\end{equation}
Let $D$ be the union of all other components. Then
\begin{equation}\label{eq:Dlarge}
|D|\ge n-z-|U|-|A|\ge n-z-q-a\ge n-\lambda-a>\frac n2,
\end{equation}
and, by \eqref{eq:RminusU},
\[
|A|+|D|>\frac{3n}{4}-1.
\]
The sets $A,D,T$ satisfy the hypotheses of Lemma~\ref{lem:separation-obstruction}; hence
\begin{equation}\label{eq:ATcograph}
G[A\cup T]\text{ is }P_4\text{-free}.
\end{equation}

We next show that $G[A\cup Z]$ is $P_4$-free. Suppose that it contains an induced $P_4$. By \eqref{eq:ATcograph} and the assumption that $G[Z]=G[P\cup T]$ is $P_4$-free, this path contains a vertex of $A$ and a vertex of $P$, and therefore at most two vertices of $T$. Since $D\subseteq B\cup W$, the stronger estimate in \eqref{eq:Dlarge} and the definition of $\mu$ give
\[
\begin{aligned}
|D\cap B|
  &\ge |D|-|W|\\
  &>(n-\lambda-a)-\left(\frac n2-\lambda-a\right)\\
  &=\frac n2.
\end{aligned}
\]
The set $D\cap B$ is anticomplete to $A$, because $A$ and $D$ are unions of different components of $R-U$, and it is anticomplete to $P$, because $P$ is anticomplete to $B$. The union of the neighborhoods of at most two vertices of $T$ has order smaller than $n/2$. Hence some vertex of $D\cap B$ is anticomplete to the induced path, a contradiction. Thus

\begin{equation}\label{eq:AZcograph}
G[A\cup Z]\text{ is }P_4\text{-free}.
\end{equation}
Define
\[
P'=P\cup A,\qquad B'=D\cap B,\qquad W'=U\cup(D\cap W).
\]
These sets, together with $T$, form a partition of $V(G)$. Every vertex of $P'\cup B'\cup W'$ has degree at least $n/4$, and $P'$ is anticomplete to $B'$. By \eqref{eq:AZcograph}, the new core $P'\cup T$ is $P_4$-free. Finally, \eqref{eq:Klower} gives
\[
|P'\cup T|=|Z|+|A|\ge|Z|+b.
\]
If $W=\varnothing$, then $W'=U$ and
\[
|W'|\le q+1\le\lambda-z+1\le\lambda+1<\mu
\]
by \eqref{eq:zq} and \eqref{eq:constants}.
\end{proof}

\section{Proof of the main theorem}

\begin{proof}[Proof of Theorem~\ref{thm:main}]
Let $G$ be a 7-tough $(P_4\cup P_1)$-free graph on $n\ge3$ vertices. Suppose, for a contradiction, that $G$ is not Hamiltonian. We may assume that $G$ is not complete. By the elementary toughness bounds and \eqref{eq:alpha},
\[
\kappa(G)\ge14,\qquad \alpha(G)\le\frac n8.
\]
If $n\le112$, then
\[
\alpha(G)\le\floor{\frac n8}\le14\le\kappa(G),
\]
and Lemma~\ref{lem:CE} gives a Hamiltonian cycle. Hence
\begin{equation}
n>112.
\end{equation}
Put
\[
h=\ceil{\frac n4},\qquad \lambda=\ceil{\frac n7},\qquad b=h-\lambda,\qquad \mu=\frac n2-\lambda-\frac n8,
\]
and define
\[
S=\{v\in V(G):d_G(v)\ge n/4\},\qquad T=V(G)\setminus S.
\]
As observed by Shan~\cite[Claim~3.1]{Shan2026},
the graph $G[T]$ is $P_4$-free.

Suppose first that $T=\varnothing$. Then $\delta(G)\ge h$. If $\alpha(G)\le2$, the bound $\kappa(G)\ge14$ gives
\[
\kappa(G)\ge\alpha(G),
\]
so Lemma~\ref{lem:CE} gives a Hamiltonian cycle. If $\alpha(G)\ge3$, then $\kappa(G)\ge14\ge2$, and Lemma~\ref{lem:HTW}, again with the empty forest, gives a cycle of order at least
\[
\min\set{n,\frac23\sigma_3(G)}.
\]
If the minimum is $n$, the cycle is Hamiltonian. Otherwise its order is at least
\[
\frac23\sigma_3(G)\ge2h\ge\lambda.
\]
All vertices outside the latter cycle have degree at least $n/4$, so Lemma~\ref{lem:long-completion} applies. Thus $T\ne\varnothing$.

Start with the partition
\[
P_0=\varnothing,\qquad B_0=S,\qquad W_0=\varnothing,\qquad Z_0=T.
\]
It satisfies conditions \textup{(i)}--\textup{(iv)} of Lemma~\ref{lem:enlargement}. If there is a cycle of order at least $\lambda$ containing every vertex of $Z_0$, then every vertex outside the cycle belongs to $S$, and Lemma~\ref{lem:long-completion} gives a Hamiltonian cycle. Hence we may assume that no such cycle exists.

Apply Lemma~\ref{lem:enlargement}. Unless $G$ is already Hamiltonian, it yields a partition
\[
V(G)=P_1\mathbin{\dot\cup}T\mathbin{\dot\cup}B_1\mathbin{\dot\cup}W_1
\]
that satisfies conditions \textup{(i)}--\textup{(iii)} of that lemma and has the $P_4$-free core
\[
Z_1=P_1\cup T,
\qquad
|Z_1|\ge|T|+b.
\]
Because the initial exceptional set is $W_0=\varnothing$, the final assertion of Lemma~\ref{lem:enlargement} also gives
\[
|W_1|<\mu,
\]
which is condition \textup{(iv)} for the new partition.

If there is a cycle of order at least $\lambda$ containing $Z_1$, Lemma~\ref{lem:long-completion} completes it, because $T\subseteq Z_1$ contains every low-degree vertex. We may therefore assume that no such cycle exists; this is precisely condition \textup{(v)}. Thus the partition $(P_1,T,B_1,W_1)$ satisfies every hypothesis of Lemma~\ref{lem:enlargement}, and the lemma may be applied a second time. Unless $G$ is Hamiltonian, the second application produces a partition satisfying conditions \textup{(i)}--\textup{(iii)} and a $P_4$-free core $Z_2=P_2\cup T$ with

\[
|Z_2|\ge|Z_1|+b\ge|T|+2b\ge2b\ge\lambda,
\]
where the last inequality is \eqref{eq:constants}.

By Lemma~\ref{lem:cograph-cover}, there is a cycle containing every vertex of $Z_2$. Its order is at least $|Z_2|\ge\lambda$. Since $Z_2$ contains $T$, every vertex outside this cycle has degree at least $n/4$. Lemma~\ref{lem:long-completion} yields a Hamiltonian cycle, the final contradiction.
\end{proof}

\begin{remark}
The constant 7 enters the proof through the two-step enlargement. Each small cut in the compressed high-degree graph enlarges the prescribed cograph core by at least
\[
b=\ceil{\frac n4}-\ceil{\frac n7}.
\]
For $n>112$, two enlargements give $2b\ge\ceil{n/7}$, the threshold required in Lemma~\ref{lem:long-completion}. The corresponding estimates at toughness 6 do not meet the asymmetric two-side criterion used here.
\end{remark}

\section*{Declaration of competing interest}
The authors declare that they have no known competing financial interests or personal relationships that could have appeared to influence the work reported in this paper.

\section*{Data availability}
No data were used for the research described in this article.

\end{document}